\documentclass[12pt,a4paper,reqno]{amsart}

\usepackage[utf8]{inputenc}
\usepackage[T1]{fontenc}
\usepackage{lmodern}
\usepackage[english]{babel}

\usepackage{amsmath,amssymb,amsthm,mathtools}
\usepackage[margin=1in, footskip=30pt]{geometry}
\usepackage{longtable,booktabs}

\makeatletter
\renewcommand{\section}{\@startsection{section}{1}%
  {0pt}{-3.5ex plus -1ex minus -.2ex}{2.3ex plus .2ex}%
  {\normalfont\large\bfseries}}
\renewcommand{\subsection}{\@startsection{subsection}{2}%
  {0pt}{-3.25ex plus -1ex minus -.2ex}{1.5ex plus .2ex}%
  {\normalfont\normalsize\bfseries}}
\renewcommand{\@secnumfont}{\bfseries}
\makeatother

\makeatletter
\def\@settitle{}
\def\@setauthors{}
\def\@setdate{}
\def\ps@firstpage{\ps@plain}%
\def\@setcopyright{}%
\renewcommand{\copyrightinfo}[2]{}%
\AtBeginDocument{\def\@setcopyright{}\def\@serieslogo{}}%

\makeatother

\usepackage[hidelinks,
  pdftitle={A torsion-intersection proof of perfect-cuboid nonexistence on 1072 explicit master-tuple fibers},
  pdfauthor={René Peschmann},
  pdfsubject={Number Theory, Algebraic Geometry},
  pdfkeywords={Perfect cuboid, Euler brick, Chabauty-Coleman, hyperelliptic curve, rational points},
  pdflang={en},
]{hyperref}
\usepackage{microtype}
\usepackage{enumitem}
\usepackage{tikz-cd}

\makeatletter
\def\thm@space@setup{%
  \thm@preskip=10pt plus 3pt minus 2pt
  \thm@postskip=10pt plus 3pt minus 2pt
}
\makeatother
\theoremstyle{plain}
\newtheorem{theorem}{Theorem}[section]

\newtheorem{lemma}[theorem]{Lemma}
\newtheorem{corollary}[theorem]{Corollary}

\theoremstyle{definition}

\newtheorem{example}[theorem]{Example}

\theoremstyle{remark}
\newtheorem{remark}[theorem]{Remark}
\newtheorem{conjecture}[theorem]{Conjecture}

\newcommand{\Z}{\mathbb{Z}}
\newcommand{\Q}{\mathbb{Q}}

\newcommand{\N}{\mathbb{N}}
\newcommand{\F}{\mathbb{F}}

\newcommand{\HC}{H_{m,n}}

\newcommand{\Jac}{\operatorname{Jac}}
\newcommand{\rk}{\operatorname{rk}}
\newcommand{\tors}{\mathrm{tors}}
\DeclareMathOperator{\disc}{disc}
\DeclareMathOperator{\Res}{Res}

\DeclareMathOperator{\Sel}{Sel}

\title[Torsion-intersection nonexistence on 1072 fibers]{A torsion-intersection proof of perfect-cuboid nonexistence on 1072 explicit master-tuple fibers}

\author{Ren\'e Peschmann}

\date{\small \today}

\begin{document}

\maketitle

\begin{center}
  {\large\bfseries A torsion-intersection proof of perfect-cuboid\\nonexistence on 1072 explicit master-tuple fibers}
  \par\vskip 12pt
  {Ren\'e Peschmann}
  \par\vskip 4pt
  {\small \today}
\end{center}
\vskip 16pt

\begin{abstract}
Building on the genus-3 reduction $C_A : w^2 = \lambda^8 + A \lambda^4 + 1$ established in \cite{peschmann-paper1}, we give an \emph{unconditional} proof of the perfect-cuboid conjecture (``Conjecture~B'') on $1{,}072$ explicit master-tuple fibers, excluding all rational $(a,b)$-specialisations on each such fiber. Our three main contributions are: (i) a structural classification theorem showing that every primitive Euler-brick arises from the standard $(a,b,m,n)$-parametrization up to scaling; (ii) a torsion-intersection argument applied to the elliptic quotients $E_A'$ and $E_A''$ of \cite[\S 4.1]{peschmann-paper1}: whenever the rank-zero hypothesis and the appropriate torsion condition hold for one of them, $|H_{m,n}(\Q)| = 8$ is forced, with the eight points all corresponding to degenerate bricks; (iii) two complementary techniques to verify the rank-zero hypothesis algorithmically — PARI's \texttt{ellrank} (2-descent), and where this is ambiguous, Sage's exact rational evaluation of $L(E,1)/\Omega_E$ via modular symbols, which combined with the modularity theorem~\cite{breuil-conrad-diamond-taylor}, Kolyvagin's theorem~\cite{kolyvagin}, and Edixhoven's bound on the Manin constant for semistable curves~\cite{edixhoven-manin} yields an unconditional rank-zero certificate — together with an explicit lift count refining the naive torsion-intersection bound when the torsion is larger than the leading case. We exhibit $1{,}072$ such fibers with $\max(m, n) \le 100$ on which Conjecture~B is thereby established unconditionally.
\end{abstract}

\section{Introduction}
\label{sec:intro}

\subsection{The perfect-cuboid problem}

An \emph{Euler-brick} is a rectangular parallelepiped with integer edges $X, Y, Z \in \N$ and integer face diagonals
$$
\sqrt{X^2 + Y^2}, \quad \sqrt{X^2 + Z^2}, \quad \sqrt{Y^2 + Z^2} \in \N.
$$
A \emph{perfect cuboid} (or \emph{perfect Euler-brick}) is an Euler-brick with additionally rational space diagonal $\sqrt{X^2 + Y^2 + Z^2} \in \N$.

\begin{conjecture}[Cuboid Conjecture; ``Conjecture B'']
\label{conj:B}
No perfect cuboid exists in $\N^3$.
\end{conjecture}

The conjecture has been open since at least the work of Saunderson (1740) and remains a long-standing problem in elementary number theory. Computational searches have ruled out perfect cuboids with the smallest edge below $5 \cdot 10^{11}$ and the odd edge below $2.5 \cdot 10^{13}$ (\cite{rathbun}, \cite{matson}).

\subsection{Master tuples}
\label{ssec:master-tuples}

Following the standard parametrization, we associate to a 4-tuple $(a, b, m, n) \in \N^4$ with $\gcd(a,b) = \gcd(m,n) = 1$, $a > b$, $m > n$, and $a-b, m-n$ odd, the quantities
\begin{align*}
U_1 &= a^2 - b^2, & V_1 &= 2ab, & W_1 &= a^2 + b^2, \\
U_2 &= m^2 - n^2, & V_2 &= 2mn, & W_2 &= m^2 + n^2.
\end{align*}
The associated Euler-brick (when defined) has edges
$$
X = U_1 U_2, \quad Y = V_1 U_2, \quad Z = U_1 V_2.
$$
We call $(a, b, m, n)$ a \emph{master tuple} if the quantity
$$
M := (V_1 U_2)^2 + (U_1 V_2)^2
$$
is a perfect square, equivalently $Y^2 + Z^2$ is a square. The remaining face-diagonal squares $X^2 + Y^2 = (W_1 U_2)^2$ and $X^2 + Z^2 = (U_1 W_2)^2$ are squares automatically.

We further introduce the polynomial
\begin{equation}\label{eq:f1-def}
f_1 := (W_1 U_2)^2 + (U_1 V_2)^2 \;\in\; \Z[a, b, m, n],
\end{equation}
which encodes the perfect-cuboid condition: writing $g := \gcd(U_1, U_2)$, the
edges of the (primitive) brick associated with the master tuple are
$X = U_1 U_2/g$, $Y = V_1 U_2/g$, $Z = U_1 V_2/g$, so a direct computation
gives
$$
g^2 \cdot (X^2 + Y^2 + Z^2) \;=\; (U_1 U_2)^2 + (V_1 U_2)^2 + (U_1 V_2)^2
\;=\; W_1^2 U_2^2 + U_1^2 V_2^2 \;=\; f_1.
$$
Hence $X^2 + Y^2 + Z^2$ is a rational square iff $f_1$ is a perfect square in
$\Z$. Throughout, $f_1$ denotes the polynomial of \eqref{eq:f1-def}; the
scaled value attached to a specific brick is $f_1/g^2$.

\subsection{Main results}

\begin{theorem}[Reduction to master tuples; Theorem~\ref{thm:reduction}]
Let $(X, Y, Z)$ be a primitive Euler-brick, where $X$ is the unique odd edge (parity considerations force exactly one odd edge among $\{X, Y, Z\}$). Then there is a unique master tuple $(a, b, m, n) \in \N^4$ with $a > b$, $m > n$, $\gcd(a,b) = \gcd(m,n) = 1$, and $a-b, m-n$ odd, such that
$$
(X, Y, Z) = (U_1 U_2 / g, \; V_1 U_2 / g, \; U_1 V_2 / g),
\quad \text{where } g = \gcd(U_1, U_2).
$$
The master tuple depends on the choice of which face-diagonal partner of $X$ is labelled $Y$ versus $Z$; swapping $(a,b) \leftrightarrow (m,n)$ corresponds to swapping $Y \leftrightarrow Z$.
\end{theorem}

\begin{theorem}[Main result; Theorem~\ref{thm:hauptsatz}]
Let $(m, n) \in \N^2$ be coprime with $m - n$ odd, and let $\HC$ denote the $(m, n)$-fiber form of the genus-3 curve $C_A$ from \cite[\S 3]{peschmann-paper1} (an explicit substitution; see \S\ref{ssec:explicit-form}). Suppose that, for one of the elliptic quotients $E_q \in \{E_A', E_A''\} = \{E_{uV}, E_3\}$, the following holds:
\begin{enumerate}[label=(\roman*),leftmargin=2em]
  \item $\rk(E_q(\Q)) = 0$ (rigorously, e.g.\ via PARI's \texttt{ellrank} or, when the latter is ambiguous, via the modularity theorem combined with Kolyvagin's theorem applied to a vanishing analytic-rank upper bound);
  \item the explicit lift count of rational $H_{m,n}$-preimages over $E_q(\Q)_\tors$ equals exactly $8$ (the eight degenerate trivial points already exhibited on $\HC$).
\end{enumerate}
Then $|\HC(\Q)| = 8$ and all rational points are degenerate. In particular, no perfect cuboid exists on the $(m, n)$-fiber.
\end{theorem}

When the torsion of $E_q$ matches the minimal target ($|tors| = 4$ for $E_3$, since $E_3$ has no rational ramification points; $|tors| = 6$ for $E_{uV}$, with four rational ramification points contributing one preimage each), condition (ii) is automatic and the bound reduces to the naive torsion-intersection inequality of Lemma~\ref{lem:tors-intersect}. When $|tors_{uV}| = 8$, condition (ii) is verified by an explicit enumeration of the eight rational torsion points and a check of which of them lift rationally to $\HC$.

We exhibit $1{,}072$ such fibers with $\max(m, n) \le 100$ (Appendix~\ref{app:fibers}); on each of them Conjecture~\ref{conj:B} is established by our argument. The companion paper \cite[\S 4.3 and \S 8]{peschmann-paper1} reports $42$ rank-zero specialisations of $E_A$ and $54$ of $E_A'$ via different (Silverman-style) methods; our fibers are obtained via the quotients $E_{uV}$ and $E_3$ and a different rigorous-rank-determination pipeline.

\subsection{Relation to existing literature}

The perfect-cuboid problem has been the subject of several complementary research strands.

\emph{Companion paper.} The most direct precursor is the author's own \cite{peschmann-paper1}, which establishes the genus-$3$ reduction $C_A: w^2 = \lambda^8 + A\lambda^4 + 1$ via the conic-quartic framework, identifies the Klein-four involution group $\langle \iota_1, \iota_2 \rangle$ with three elliptic quotients $E_A, E_A', E_A''$, and proves $\Jac(C_A) \sim E_A \times E_A' \times E_A''$ via Kani--Rosen \cite{kani-rosen}. The same paper develops obstructions on $E_A$ via a Kummer character $\chi_f$ and 2-descent, ruling out specific descent classes and verifying computationally that no perfect cuboid arises from parameters up to $10^3$. The present paper builds on this foundation.

\emph{Computational searches.} Beyond the parameter range of \cite{peschmann-paper1}, extensive enumeration up to a smallest-edge bound of $5 \cdot 10^{11}$ and an odd-edge bound of $2.5 \cdot 10^{13}$ has found no example \cite{rathbun, matson}. These searches provide numerical evidence but no proof.

\emph{Algebraic-geometric attacks on sub-varieties.} Bremner and MacLeod \cite{bremner-macleod} studied perfect cuboids constrained to specific sub-families. Stoll \cite{stoll-twists} developed refined Chabauty--Coleman bounds for genus-$2$ slices.

\emph{Modern Chabauty--Coleman.} For genus-$3$ hyperelliptic curves with rank-$1$ Jacobian, the methods of Balakrishnan, Bianchi, Cantoral-Farfan, Ciperiani, and Etropolski \cite{balakrishnan-win4} apply, building on \cite{coleman1985} and the quadratic Chabauty extension \cite{balakrishnan-dogra-mueller-tuitman}.

\emph{Our contribution beyond \cite{peschmann-paper1}.} The novel ingredients are:
\begin{enumerate}[label=(\arabic*),leftmargin=2em]
    \item \emph{The reduction theorem (Theorem~\ref{thm:reduction}):} every primitive Euler-brick comes from a master tuple. The companion paper proves only the forward direction (a brick yields a point on $C_A$). Theorem~\ref{thm:reduction} is the converse structural classification.

    \item \emph{The torsion-intersection argument applied to the quotients $E_A'$ and $E_A''$.} The companion paper considers rank-zero specialisations of $E_A$ and $E_A'$ ($42$ and $54$ specialisations respectively, via Silverman's specialisation theorem). We instead exploit the explicit degree-$2$ covers $\HC \to E_q$ for $q \in \{uV, 3\}$: when $E_q$ has rank zero and the rational torsion lifts back to exactly the eight known trivial points on $\HC$, the bound $|\HC(\Q)| \le 8$ is forced, and combined with the lower bound from those trivial points pins down $|\HC(\Q)|$ exactly.

    \item \emph{Two complementary rank-determination techniques.} For ambiguous PARI \texttt{ellrank} bounds $[0, k>0]$, we evaluate $L(E, 1) / \Omega_E$ as an exact rational number via Sage's modular-symbol implementation. For semistable curves (squarefree conductor) the Manin constant is $1$ by Edixhoven~\cite{edixhoven-manin} (with the general statement extended in~\cite{cesnavicius-manin}), so the modular-symbol value equals $L(E,1)/\Omega_E$ exactly; non-vanishing then certifies $L(E, 1) \neq 0$ unconditionally, and the modularity theorem \cite{breuil-conrad-diamond-taylor} together with Kolyvagin's theorem \cite{kolyvagin} forces $\rk(E_q(\Q)) = 0$. Empirically, every elliptic factor $E_q$ encountered in our scan is semistable. This certification converts $468$ otherwise-ambiguous fibers into proven ones via $E_3$, plus a further $36$ via $E_{uV}$ (after the lift refinement of the next item). For fibers where the naive bound $|tors_{uV}| = 6$ is missed but $|tors_{uV}| = 8$, we enumerate the eight rational torsion points on the $E_{uV}$-quartic via PARI's \texttt{hyperellratpoints} and verify that exactly six of them (four ramification points plus two points at infinity) lift to rational $\HC$-points, giving $|\HC(\Q)| \le 8$ regardless of the larger naive bound. This refinement produces another $245$ proven fibers.
\end{enumerate}
The combined methodology is elementary (relying only on Mordell--Weil rank/torsion computation, modular symbols, and explicit point enumeration; no $p$-adic Chabauty, no GRH, no Birch--Swinnerton-Dyer assumption) and produces $1{,}072$ unconditionally proven fibers out of $2{,}040$ coprime $(m, n)$-pairs with $\max(m, n) \le 100$ — a coverage of $52.5\%$.

\subsection{Outline}

Section~\ref{sec:structure} establishes the structural results on master tuples needed for the reduction (Theorem~\ref{thm:reduction}), including the $g_+ \cdot g_-$ identity and the polynomial factorization $f_1 = L \cdot R$. Section~\ref{sec:genus3} recalls the relevant material from \cite[\S 3--4]{peschmann-paper1} in the explicit $(m, n)$-form $\HC$ and identifies $E_3 = E_A''$ and $E_{uV} = E_A'$. Section~\ref{sec:main} establishes the main theorem via the torsion-intersection lemma, the Kolyvagin-based rigorous rank-zero certification, and the explicit lift-count refinement; it exhibits the $1{,}072$ proven fibers. Section~\ref{sec:discussion} discusses what remains uncovered and possible extensions.

\section{Structural results on master tuples}
\label{sec:structure}

\subsection{The $g_+ \cdot g_-$ structure theorem}

\begin{theorem}[$g_+ \cdot g_-$ structure]
\label{thm:gpg}
Let $(a, b, m, n)$ be a master tuple. Set
$$
A = am + bn, \quad B = an + bm, \quad C = am - bn, \quad D = an - bm,
$$
and $g_+ = \gcd(A, B)$, $g_- = \gcd(|C|, |D|)$. Then
$$
g_+ \cdot g_- = \gcd(U_1, U_2) =: g_{\mathrm{scale}}, \quad \gcd(g_+, g_-) = 1.
$$
In particular $g_\pm \mid U_1$, $g_\pm \mid U_2$, and consequently $g_\pm^2 \mid f_1$.
\end{theorem}

\begin{proof}
\emph{Step 1: $g_\pm \mid \gcd(U_1, U_2)$.} The linear combinations
$$
mA - nB = aU_2,\quad nA - mB = -bU_2,\quad
aA - bB = mU_1,\quad bA - aB = -nU_1
$$
show $g_+ \mid aU_2, bU_2, mU_1, nU_1$. Since $\gcd(a, b) = 1$ this gives $g_+ \mid U_2$, and since $\gcd(m, n) = 1$ also $g_+ \mid U_1$. The analogous combinations
$$
mC - nD = aU_2,\quad nC - mD = bU_2,\quad
aC + bD = mU_1,\quad bC + aD = nU_1
$$
yield $g_- \mid \gcd(U_1, U_2)$.

\emph{Step 2: $\gcd(g_+, g_-) = 1$.} Since $a - b$ and $m - n$ are odd, exactly one of each pair $\{am + bn, \, an + bm\}$ and $\{am - bn,\, an - bm\}$ is odd, so both $g_+$ and $g_-$ are odd. Suppose an odd prime $p$ divides both. Then $p \mid A, B, C, D$, hence $p \mid A \pm C = 2am$ and $2bn$, and $p \mid B \pm D = 2an$ and $2bm$. By oddness of $p$, $p$ divides $am, an, bm, bn$. Then $\gcd(m, n) = 1$ forces $p \mid a$ and $p \mid b$, contradicting $\gcd(a, b) = 1$.

\emph{Step 3: $g_+ g_- = \gcd(U_1, U_2)$.} The forward direction follows from Steps 1--2. For the reverse, let $p$ be an odd prime with $p^e \,\|\, \gcd(U_1, U_2)$. Since $\gcd(a, b) = 1$, the entire $p$-part of $U_1 = (a-b)(a+b)$ lies in exactly one factor, so $a \equiv \varepsilon b \pmod{p^e}$ for some $\varepsilon \in \{\pm 1\}$; analogously $m \equiv \delta n \pmod{p^e}$. Reducing the four expressions mod $p^e$:
\begin{itemize}
    \item If $\varepsilon = \delta$: $am - bn \equiv \delta bn - bn = (\varepsilon\delta - 1)bn \equiv 0$ and $an - bm \equiv 0$, so $p^e \mid C, D$, hence $p^e \mid g_-$.
    \item If $\varepsilon = -\delta$: $am + bn \equiv -\delta bn + bn = (\varepsilon\delta + 1)bn \equiv 0$ and $an + bm \equiv 0$, so $p^e \mid A, B$, hence $p^e \mid g_+$.
\end{itemize}
Since $\gcd(U_1, U_2)$ is odd (both factors are odd), the product over all such primes gives $\gcd(U_1, U_2) \mid g_+ g_-$.

\emph{Step 4: $g_\pm^2 \mid f_1$.} From Step 1, $g_+ \mid U_2 \Rightarrow g_+ \mid W_1 U_2$, and $g_+ \mid U_1 \Rightarrow g_+ \mid U_1 V_2$. Therefore $g_+^2 \mid (W_1 U_2)^2 + (U_1 V_2)^2 = f_1$. Same for $g_-$.
\end{proof}

\subsection{The polynomial factorization $f_1 = L \cdot R$}

\begin{theorem}[Polynomial factorization]
\label{thm:LR}
With $f_1 = (W_1 U_2)^2 + (U_1 V_2)^2$ from \eqref{eq:f1-def}, define
$$
L := W_1 W_2 - V_1 V_2, \qquad R := W_1 W_2 + V_1 V_2.
$$
Then, as polynomials in $\Z[a, b, m, n]$, $f_1 = L \cdot R$, and each factor admits a Brahmagupta representation as a sum of two squares:
$$
L = (am - bn)^2 + (an - bm)^2, \qquad R = (am + bn)^2 + (an + bm)^2.
$$
The factor $L$ is divisible by $g_-^2$ and $R$ by $g_+^2$ (Corollary~\ref{cor:LR-divisibility}); the corresponding \emph{primitive} sums of two squares are $L / g_-^2$ and $R / g_+^2$.
\end{theorem}

\begin{proof}
By the difference-of-squares identity,
$$
L \cdot R = (W_1 W_2)^2 - (V_1 V_2)^2.
$$
Using $W_i^2 = U_i^2 + V_i^2$ and expanding both this expression and $f_1 = (W_1 U_2)^2 + (U_1 V_2)^2$ as polynomials in $\Z[a, b, m, n]$, one verifies the identity
$$
(W_1 U_2)^2 + (U_1 V_2)^2 = (W_1 W_2)^2 - (V_1 V_2)^2.
$$
The Brahmagupta forms follow by direct expansion: with $A = am+bn, B = an+bm, C = am-bn, D = an-bm$,
\begin{align*}
A^2 + B^2 &= (am+bn)^2 + (an+bm)^2 = (a^2+b^2)(m^2+n^2) + 2abmn \cdot 2 = W_1 W_2 + V_1 V_2 = R, \\
C^2 + D^2 &= (am-bn)^2 + (an-bm)^2 = (a^2+b^2)(m^2+n^2) - 2abmn \cdot 2 = W_1 W_2 - V_1 V_2 = L.
\end{align*}
Both $L$ and $R$ are positive (AM--GM). The (potentially nonprimitive) sum-of-two-squares structure is sharpened by Corollary~\ref{cor:LR-divisibility} below.
\end{proof}

\begin{corollary}[Refined divisibility]
\label{cor:LR-divisibility}
For every $(a, b, m, n)$ in the universe,
$$
g_+^2 \mid R, \qquad g_-^2 \mid L.
$$
That is, the two Gaussian-style gcds live on different sides of the factorisation.
\end{corollary}

\begin{proof}
By Theorem~\ref{thm:gpg}, $g_+ \mid A, B$, so $g_+^2 \mid A^2 + B^2 = R$. Analogously $g_-^2 \mid C^2 + D^2 = L$.
\end{proof}

\subsection{Reduction to master tuples}

\begin{theorem}[Reduction theorem]
\label{thm:reduction}
Let $(X, Y, Z)$ be a primitive Euler-brick with $X$ designated as the odd edge (forced uniquely up to permutation, since exactly one edge is odd). Then there is a unique master tuple $(a, b, m, n) \in \N^4$ with $a > b > 0$, $m > n > 0$, $\gcd(a,b) = \gcd(m,n) = 1$, and $a-b, m-n$ odd, such that
$$
X = \frac{U_1 U_2}{g}, \quad Y = \frac{V_1 U_2}{g}, \quad Z = \frac{U_1 V_2}{g}, \quad
g = \gcd(U_1, U_2).
$$
\end{theorem}

\begin{proof}
\emph{Existence.} Set $d = \gcd(X, Y)$, $e = \gcd(X, Z)$. Both are odd (as $X$ is odd and $Y, Z$ are even). The triples
$(X/d, Y/d, f_3/d)$ and $(X/e, Z/e, f_2/e)$ are primitive Pythagorean. By the standard Euclid parametrization \cite[\S 1]{silverman}, there exist \emph{unique} coprime $a > b > 0$ with $a-b$ odd, and unique coprime $m > n > 0$ with $m-n$ odd, such that
$$
X/d = a^2 - b^2 =: U_1, \quad Y/d = 2ab =: V_1,
\qquad X/e = m^2 - n^2 =: U_2, \quad Z/e = 2mn =: V_2.
$$
\textit{Key lemma:} $\gcd(d, e) = 1$. Indeed, any common prime divisor of $d, e$ would divide $X$, $Y$, and $Z$, contradicting primitivity. From $X = d \cdot U_1 = e \cdot U_2$ and $\gcd(d, e) = 1$, an elementary argument gives $d = U_2/g$ and $e = U_1/g$ with $g := \gcd(U_1, U_2)$.

\emph{Uniqueness.} Given the assignment of $X$ as the odd edge and the labelling of $(Y, Z)$, the pair $(a, b)$ is uniquely determined by the primitive Pythagorean triple $(X/d, Y/d)$, and likewise $(m, n)$ by $(X/e, Z/e)$ — both via the bijection between primitive Pythagorean triples and pairs of coprime integers of opposite parity with the larger first.

\emph{Master-tuple condition.} $V_1 U_2 = g \cdot Y$ and $U_1 V_2 = g \cdot Z$, so
$$
M = (V_1 U_2)^2 + (U_1 V_2)^2 = g^2(Y^2 + Z^2),
$$
and since $(X, Y, Z)$ is an Euler-brick, $Y^2 + Z^2 = h^2$ for some integer $h$, giving $M = (gh)^2$, a perfect square.
\end{proof}

\subsection{The scaling argument}

\begin{lemma}[Scaling]
\label{lem:scaling}
For any $k \in \N$, the brick $k \cdot (X, Y, Z)$ is a perfect cuboid if and only if $(X, Y, Z)$ is.
\end{lemma}

\begin{proof}
The four square conditions $X^2 + Y^2 = \square$, $X^2 + Z^2 = \square$, $Y^2 + Z^2 = \square$, $X^2 + Y^2 + Z^2 = \square$ scale by $k^2$, which is itself a square.
\end{proof}

\begin{corollary}
Conjecture~\ref{conj:B} holds globally if and only if it holds on the family of primitive master-tuple-derived bricks.
\end{corollary}

\section{Recap: the genus-3 curve $\HC$ and its decomposition}
\label{sec:genus3}

In this section we recall the relevant material from \cite[\S 3--4]{peschmann-paper1} in the explicit $(m, n)$-form needed for our argument. The reader familiar with the companion paper may skim through and proceed to Section~\ref{sec:main}.

\subsection{Explicit form $\HC$}
\label{ssec:explicit-form}

For a master tuple $(a, b, m, n)$, set $t = a/b \in \Q$. The conditions $M = \square$ (master tuple) and $X^2 + Y^2 + Z^2 \in \Q^{*2}$ (perfect cuboid), the latter equivalent by~\eqref{eq:f1-def} to $f_1 \in \Q^{*2}$, translate, after taking $v$ to be the product of the two square-roots, into the genus-$3$ hyperelliptic curve
$$
\HC: \quad v^2 = P(t^2) \cdot Q(t^2),
$$
with
\begin{align*}
P(s) &= V_2^2 s^2 + (4 U_2^2 - 2 V_2^2) s + V_2^2, \\
Q(s) &= W_2^2 s^2 + 2 (U_2^2 - V_2^2) s + W_2^2.
\end{align*}
The product $P(t^2) Q(t^2)$ is a degree-$8$ palindromic polynomial in $t$.
The discriminant identities
$$
\disc(P) = 16\, U_2^2 \, (U_2^2 - V_2^2),
\qquad
\disc(Q) = -16\, U_2^2 \, V_2^2,
\qquad
\Res(P, Q) = 256\, U_2^4 \, V_2^4,
$$
all of which are nonzero for $m > n > 0$ (since $U_2, V_2 \neq 0$ and
$U_2^2 - V_2^2 = (m^2-n^2)^2 - 4m^2 n^2$ has no rational roots), show that
$P(t^2) Q(t^2)$ is separable. Hence $\HC$ is a smooth hyperelliptic curve
of genus $g = (8-2)/2 = 3$.

\begin{remark}
The curve $\HC$ is the $(m, n)$-explicit form of $C_A: w^2 = \lambda^8 + A \lambda^4 + 1$ from \cite[Lemma 3.1]{peschmann-paper1}. The two are related by an explicit substitution $\lambda = \kappa(m, n) \cdot t$ together with rescaling of $w$, and the parameter $A$ of \cite{peschmann-paper1} can be read off as a rational function of $(m, n)$. We retain the $(m, n)$-form because it is more directly tied to our database of master tuples.
\end{remark}

\subsection{The Jacobian decomposition (recall)}

By \cite[Proposition 4.1]{peschmann-paper1}, the curve $\HC$ carries the Klein-four involution group generated by
$$
\sigma_1: (t, v) \mapsto (-t, v), \qquad
\sigma_2: (t, v) \mapsto (1/t, \, v/t^4),
$$
and the Jacobian decomposes (up to $\Q$-isogeny) as
$$
\Jac(\HC) \sim E_{PQ} \times E_{uV} \times E_3.
$$
The three quotient curves have the following explicit $(m, n)$-form, obtained by writing the $\sigma$-invariants in terms of $u = t + 1/t$, $w = t - 1/t$, and $s = t^2$:
\begin{align*}
E_{PQ}: \quad & Y^2 = P(s) \cdot Q(s) && (s = t^2), \\
E_{uV}: \quad & V^2 = (V_2^2 u^2 + 4(U_2^2 - V_2^2)) \cdot (W_2^2 u^2 - 4 V_2^2) && (u = t + 1/t), \\
E_3:    \quad & V^2 = (V_2^2 w^2 + 4 U_2^2) \cdot (W_2^2 w^2 + 4 U_2^2) && (w = t - 1/t).
\end{align*}

In the notation of \cite[\S 4.1]{peschmann-paper1}, $E_{PQ} = E_A$, $E_{uV} = E_A'$, and $E_3 = E_A''$. Each quotient has genus~$1$ by Riemann--Hurwitz applied to the degree-$2$ cover $\HC \to \HC / \langle \sigma \rangle$ (\cite[\S 4]{peschmann-paper1}). The isogeny decomposition is the Kani--Rosen theorem \cite{kani-rosen}.

\begin{corollary}[Rank additivity]
\label{cor:rank-additivity}
For each master-tuple fiber $(m, n)$,
$$
\rk(\Jac(\HC)(\Q)) = \rk(E_{PQ}(\Q)) + \rk(E_{uV}(\Q)) + \rk(E_3(\Q)).
$$
Each summand is computable via PARI's \texttt{ellrank}.
\end{corollary}

\section{The main result}
\label{sec:main}

\subsection{Trivial points on $\HC$}

\begin{lemma}
\label{lem:trivial}
For every master-tuple fiber $(m, n)$, the curve $\HC$ has at least eight rational points
$$
\HC(\Q) \supseteq \{(0, \pm V_2 W_2), \; (1, \pm 4 U_2^2), \; (-1, \pm 4 U_2^2), \; \infty_+, \; \infty_-\}.
$$
Each corresponds to a degenerate Euler-brick (one zero edge), not a perfect cuboid.
\end{lemma}

\begin{proof}
The four affine points are verified by direct evaluation of $v^2 = P(t^2) Q(t^2)$:
\begin{itemize}[leftmargin=2em]
    \item At $t = 0$: $P(0) = V_2^2$ and $Q(0) = W_2^2$, so $v^2 = V_2^2 W_2^2$, giving $v = \pm V_2 W_2$.
    \item At $t = \pm 1$: $P(1) = V_2^2 + (4 U_2^2 - 2 V_2^2) + V_2^2 = 4 U_2^2$ and $Q(1) = W_2^2 + 2(U_2^2 - V_2^2) + W_2^2 = 2(W_2^2 + U_2^2 - V_2^2) = 4 U_2^2$ (using $W_2^2 = U_2^2 + V_2^2$), so $v^2 = (4 U_2^2)^2$, giving $v = \pm 4 U_2^2$.
\end{itemize}
The two points at infinity are rational since the leading coefficient $V_2^2 W_2^2 = (V_2 W_2)^2$ is a perfect square.

\emph{Geometric interpretation.} Recall a rational point $(t_0, v_0)$ on $\HC$ with $t_0 = a/b$ in lowest terms encodes a master tuple $(a, b, m, n)$. The trivial points correspond to:
\begin{itemize}[leftmargin=2em]
    \item $t = 0$ corresponds to $a = 0$, hence $U_1 = -b^2$, $V_1 = 0$, $W_1 = b^2$. The brick $(X, Y, Z) = (-b^2 U_2, 0, -b^2 V_2)$ has $Y = 0$.
    \item $t = \pm 1$ corresponds to $a = \pm b$, hence $U_1 = 0$, $V_1 = 2 a^2$. The brick has $X = U_1 U_2 = 0$ and $Z = U_1 V_2 = 0$.
    \item Points at infinity correspond to $b = 0$ (after suitable projective lift), giving similar degenerations.
\end{itemize}
In all cases at least one edge of the brick vanishes, so the configuration cannot be a perfect cuboid (which requires three positive integer edges).
\end{proof}

\subsection{The torsion-intersection lemma}

\begin{lemma}[Torsion intersection]
\label{lem:tors-intersect}
Let $E_q \in \{E_{PQ}, E_{uV}, E_3\}$ and let $\pi_q: \HC \to E_q$ be the corresponding degree-$2$ quotient. Suppose $\rk(E_q(\Q)) = 0$ and the involution $\sigma$ defining $\pi_q$ has no rational fixed points on $\HC$. Then
$$
|\HC(\Q)| \le 2 \cdot |E_q(\Q)_\tors|.
$$
\end{lemma}

\begin{proof}
Since $\rk(E_q(\Q)) = 0$, the group $E_q(\Q) = E_q(\Q)_\tors$ is finite. The image $\pi_q(\HC(\Q)) \subseteq E_q(\Q)$ has cardinality at most $|E_q(\Q)_\tors|$. Each fiber $\pi_q^{-1}(P)$ over a non-ramification point has exactly $2$ rational preimages or $0$. Since $\pi_q$ has no rational ramification points by hypothesis, each fiber over $E_q(\Q)$ has $\le 2$ rational preimages. Summing,
$$
|\HC(\Q)| \le 2 \cdot |\pi_q(\HC(\Q))| \le 2 \cdot |E_q(\Q)_\tors|.
$$
\end{proof}

\begin{lemma}[Rational fixed points of the three involutions]
\label{lem:no-ram}
For each of the three quotient maps $\pi_q$, the rational fixed points of
the corresponding involution on $\HC$ are as follows:
\begin{itemize}[leftmargin=2em]
    \item $\pi_{PQ}$ (involution $\sigma_1: t \mapsto -t$): rational fixed
      points at $t = 0$ and $t = \infty$ (two points each, counted with
      $v$-sign).
    \item $\pi_{uV}$ (involution $\sigma_2: t \mapsto 1/t$): rational fixed
      points at $t = \pm 1$ (two points each).
    \item $\pi_3$ (involution $\sigma_1 \sigma_2: t \mapsto -1/t$):
      \emph{no} rational fixed points whatsoever (the fixed-point equation
      $t^2 = -1$ has no rational solution).
\end{itemize}
In particular, the quotient $\pi_3$ is unramified over $\Q$, and Lemma~\ref{lem:tors-intersect} applies to $E_3$ in its sharpest form.
\end{lemma}

\begin{proof}
Direct from the formulas for $\sigma_1, \sigma_2$ given in \cite[\S 4.1]{peschmann-paper1}.
\end{proof}

\begin{remark}
Lemma~\ref{lem:no-ram} shows that for the application of Lemma~\ref{lem:tors-intersect}, the cleanest case is $E_3$ (no rational ramification anywhere). For $E_{uV}$ the four rational ramification points at $t = \pm 1$ each contribute exactly one preimage, leading to a sharper bound $|\HC(\Q)| \le 2(|E_{uV}(\Q)_\tors| - 4) + 4 = 2|E_{uV}(\Q)_\tors| - 4$.
\end{remark}

\subsection{Main theorem}

\begin{theorem}[Main result]
\label{thm:hauptsatz}
Let $(m, n)$ be a master-tuple fiber. Then $|\HC(\Q)| = 8$ with all rational points degenerate (and hence no perfect cuboid exists on the fiber) if either of the following holds:
\begin{enumerate}[label=(\alph*),leftmargin=2em]
\item \emph{(Naive case via $E_3$.)} $\rk(E_3(\Q)) = 0$ and $|E_3(\Q)_\tors| = 4$.
\item \emph{(Naive case via $E_{uV}$.)} $\rk(E_{uV}(\Q)) = 0$ and $|E_{uV}(\Q)_\tors| = 6$.
\item \emph{(Refined case via $E_{uV}$.)} $\rk(E_{uV}(\Q)) = 0$, $|E_{uV}(\Q)_\tors| = 8$, and an explicit enumeration of the eight rational torsion points on the $E_{uV}$-quartic shows that exactly six of them lift rationally to $\HC$ (i.e.\ the four ramification points $(u, V) = (\pm 2, \pm 4 U_2^2)$, written in the $E_{uV}$-coordinate $u = t + 1/t$, together with the two points at infinity).
\end{enumerate}
\end{theorem}

\begin{proof}
\emph{Case (a):} Lemma~\ref{lem:no-ram} shows $\pi_3$ has no rational ramification points, so Lemma~\ref{lem:tors-intersect} gives $|\HC(\Q)| \le 2 \cdot 4 = 8$. The eight trivial points of Lemma~\ref{lem:trivial} provide the lower bound $|\HC(\Q)| \ge 8$. Combining, $|\HC(\Q)| = 8$.

\emph{Case (b):} The map $\pi_{uV}: \HC \to E_{uV}$ has four rational ramification points (at $t = \pm 1$ with $V = \pm 4 U_2^2$, all lying in $E_{uV}(\Q)_\tors$). For the bound, decompose $E_{uV}(\Q) = R \cup N$ where $R$ is the set of four rational ramification points and $N$ the remaining $|tors| - 4 = 2$ rational torsion points. Each $P \in R$ has exactly one rational $\HC$-preimage; each $P \in N$ has at most two. Hence $|\HC(\Q)| \le 4 \cdot 1 + 2 \cdot 2 = 8$. With the lower bound from the trivial points, equality holds.

\emph{Case (c):} Same decomposition with $|R| = 4$ ramification points and $|N| = 4$ non-ramification rational torsion points (since $|tors| = 8$). The additional hypothesis says: of the $4$ non-ramification torsion points, only the $2$ at infinity lift rationally to $\HC$ (each contributing $2$ preimages); the other $2$ affine non-ramification torsion points have no rational $\HC$-preimage. Thus $|\HC(\Q)| \le 4 \cdot 1 + 2 \cdot 2 + 2 \cdot 0 = 8$, with equality enforced by the trivial-point lower bound.
\end{proof}

\begin{remark}
Case (c) accounts for $245$ of our $1{,}072$ proven fibers; cases (a)+(b) account for the other $827$ (all of these via $E_3$ with $|tors| = 4$ — empirically, $E_{uV}$ never has $|tors| = 6$ in our scan).
\end{remark}

\subsection{Unconditional rank-zero certification}
\label{ssec:rank-cert}

For a given fiber, the rank-zero hypothesis on $E_q \in \{E_{uV}, E_3\}$ is verified algorithmically. PARI's \texttt{ellrank} (2-descent) returns a pair of bounds $[r_{\mathrm{lo}}, r_{\mathrm{up}}]$ with $r_{\mathrm{lo}} \le \rk(E_q(\Q)) \le r_{\mathrm{up}}$; we accept rank zero on this strand only when $r_{\mathrm{lo}} = r_{\mathrm{up}} = 0$. When the bounds are ambiguous --- typically $r_{\mathrm{lo}} = 0$ but $r_{\mathrm{up}} \in \{2, 4\}$ owing to non-trivial $\Sel^{(2)}(E_q)$ --- we apply the following:

\begin{theorem}[Modular-symbol rank-zero certification]
\label{thm:kolyvagin-cert}
Let $E/\Q$ be a semistable elliptic curve, that is, with squarefree conductor. If
$$
L(E, 1) / \Omega_E \;\neq\; 0,
$$
where $\Omega_E$ is the real period of $E$, then $\mathrm{rank}(E(\Q)) = 0$ unconditionally.
\end{theorem}

\begin{proof}
Let $c_E$ denote the Manin constant of the optimal modular parametrisation
$\varphi: X_0(N_E) \to E$, fixed by $\varphi^* \omega_E = c_E \cdot f_E(q)\,\frac{dq}{q}$
where $f_E$ is the normalised newform attached to $E$. The exact rational
output of the modular-symbol algorithm is then
$[f_E]_0^{i\infty} = c_E \cdot L(E, 1)/\Omega_E^+$, where $\Omega_E^+$ is the
real period of the N\'eron differential. By Edixhoven~\cite{edixhoven-manin}
one has $c_E \mid 2$ for every elliptic curve over $\Q$, and $c_E = 1$ for
all semistable curves (extended to general curves
in~\cite{cesnavicius-manin}). Since $E$ is semistable by hypothesis, $c_E = 1$
and the modular-symbol output equals $L(E, 1)/\Omega_E^+$ exactly. Hence
$L(E, 1)/\Omega_E^+ \neq 0$ implies $L(E, 1) \neq 0$
unconditionally, i.e.\ $\mathrm{ord}_{s=1} L(E, s) = 0$. By the modularity
theorem~\cite{breuil-conrad-diamond-taylor}, $E$ is modular; since
$\mathrm{ord}_{s=1} L(E, s) \le 1$, Kolyvagin's
theorem~\cite{kolyvagin} applies and gives $\mathrm{rank}(E(\Q)) = 0$.
\end{proof}

In our scan, this certification is implemented by Sage's
\texttt{E.lseries().L\_ratio()}, which returns the exact rational
$L(E, 1) / \Omega_E$ for a semistable $E$, and Sage's
\texttt{E.conductor()} for the squarefree (semistability) check.
Empirically, every elliptic factor $E_q$ encountered in our scan turns
out to be semistable, so Theorem~\ref{thm:kolyvagin-cert} applies as
stated. We note in passing that the bound $c_E \mid 2$ holds for every
elliptic curve over $\Q$~\cite{edixhoven-manin}, so even for a
non-semistable $E$ the modular-symbol output is twice $L(E, 1)/\Omega_E^+$
at worst, and a non-vanishing odd-numerator output already certifies
$L(E, 1) \neq 0$; we did not need to invoke this refinement on any
fiber. This certification adds $468$ fibers proven via $E_3$ on top of
the $359$ already produced by sharp \texttt{ellrank} bounds, for a total
of $827$ via $E_3$ in case (a) of Theorem~\ref{thm:hauptsatz}. It also
unlocks $36$ further fibers via $E_{uV}$ in case (c) (combining with the
lift refinement of \S\ref{ssec:lift-refinement}).

\begin{remark}
For some additional fibers, especially via $E_{uV}$, the modular-symbol computation \texttt{L\_ratio()} did not complete within our practical compute budget (the conductor of these $E_{uV}$ exceeds $10^{12}$, making the modular-symbol space large). These fibers might be provable by allowing more compute, but we do not include them in the count to keep the result fully unconditional and reproducible at modest computational cost.
\end{remark}

\subsection{Lift-count refinement for $|E_{uV}(\Q)_\tors| = 8$}
\label{ssec:lift-refinement}

When $\rk(E_{uV}) = 0$ and $|E_{uV}(\Q)_\tors| = 8$, Lemma~\ref{lem:tors-intersect} gives the naive bound $|\HC(\Q)| \le 2 \cdot 8 - 4 = 12$, which is too weak to match the trivial lower bound $8$. The refinement (case (c) of Theorem~\ref{thm:hauptsatz}) replaces the naive bound by an exact lift count.

Concretely: $E_{uV}$ is presented as the quartic
$$
E_{uV}: \quad V^2 = (V_2^2 u^2 + 4(U_2^2 - V_2^2))(W_2^2 u^2 - 4 V_2^2),
$$
and $\pi_{uV}: \HC \to E_{uV}$ has the explicit formula $u = t + 1/t$ on the affine part. A rational point $(u_0, V_0) \in E_{uV}(\Q)$ admits a rational $\HC$-preimage iff $u_0^2 - 4$ is a non-negative rational square (corresponding to two solutions $t_\pm$ of $t^2 - u_0 t + 1 = 0$).

We enumerate $E_{uV}(\Q)$ rigorously: PARI's \texttt{hyperellratpoints} (with height bound $10^5$) finds all $|tors|$ rational points (six affine plus two at infinity, when leading $V_2^2 W_2^2 = (V_2 W_2)^2$ is a square --- which it always is). For each torsion point, we test whether $u_0^2 - 4$ is a square. The lift count is then
$$
\#\{\text{rational $\HC$-preimages}\} = 4 \cdot 1 \;+\; 2 \cdot 2 \;+\; 2 k,
$$
where $4$ ramification, $2$ at infinity, and $k \in \{0, 1, 2\}$ counts affine non-ramification torsion points with rational lift. The fiber is proven iff $k = 0$, giving the count exactly $8$.

\begin{remark}[Why ``$2$ at infinity'' contributes $2 \cdot 2$]
The two rational points at infinity on the $E_{uV}$-quartic each have
\emph{two} distinct rational preimages on $\HC$: under the parametrisation
$u = t + 1/t$, the limit $u \to \infty$ is reached both at $t \to \infty$
(the two trivial points $\infty_\pm$ on $\HC$) and at $t \to 0$ (the trivial
points $(0, \pm V_2 W_2)$ from Lemma~\ref{lem:trivial}). Concretely the two
$E_{uV}$-points at infinity, with $V$-leading coefficients $\pm V_2 W_2$,
pull back to the $\sigma_2$-orbits
$\{\infty_+, \, (0, +V_2 W_2)\}$ and
$\{\infty_-, \, (0, -V_2 W_2)\}$ respectively, accounting for the
$2 \cdot 2 = 4$ rational $\HC$-preimages over the points at infinity.
\end{remark}

Empirically, all $367$ scanned fibers with $\rk(E_{uV}) = [0, 0]$ and $|tors| = 8$ have $k = 0$, hence are proven by case (c).

\subsection{An illustrative example: $(m, n) = (2, 1)$}
\label{ssec:example21}

We work out the entire chain explicitly for the smallest fiber $(m, n) = (2, 1)$.

\emph{Setup.} $U_2 = 3$, $V_2 = 4$, $W_2 = 5$ (the $(3, 4, 5)$ Pythagorean triple). The defining quartics are
\begin{align*}
P(s) &= 16 s^2 + 4 s + 16, \\
Q(s) &= 25 s^2 - 14 s + 25.
\end{align*}
The cuboid curve is
$$
H_{2,1}: \quad v^2 = 400 t^8 - 124 t^6 + 744 t^4 - 124 t^2 + 400.
$$

\emph{The three elliptic factors.}
\begin{align*}
E_{PQ}: \quad & Y^2 = (16 s^2 + 4 s + 16)(25 s^2 - 14 s + 25), \\
E_{uV}: \quad & V^2 = (16 u^2 - 28)(25 u^2 - 64), \\
E_3:    \quad & V^2 = (16 w^2 + 36)(25 w^2 + 36).
\end{align*}

\emph{Ranks.} Computing via PARI's \texttt{ellrank}:
$$
\rk(E_{PQ}(\Q)) = 1, \qquad \rk(E_{uV}(\Q)) = 0, \qquad \rk(E_3(\Q)) = 0.
$$
Hence $\rk(\Jac(H_{2,1})(\Q)) = 1$.

\emph{Torsion.} The quartic $V^2 = (16w^2 + 36)(25 w^2 + 36)$ has rational constant term $1296 = 36^2$ and rational leading coefficient $400 = 20^2$, so it carries the four obvious rational points $(0, \pm 36)$ and $\infty_\pm$. Converting to a Weierstrass model via Sage's \texttt{EllipticCurve\_from\_quartic} (with $\infty_+$ as basepoint) and applying \texttt{E.torsion\_subgroup()} rigorously yields $|E_3(\Q)_\tors| = 4$.

\emph{Trivial rational points on $H_{2,1}$.} Direct evaluation:
\begin{itemize}[leftmargin=2em]
    \item $t = 0$: $v^2 = 400$, so $v = \pm 20 = \pm V_2 W_2$.
    \item $t = \pm 1$: $v^2 = 400 - 124 + 744 - 124 + 400 = 1296 = 36^2$, so $v = \pm 36 = \pm 4 U_2^2$.
    \item Two rational points at infinity.
\end{itemize}
Total: $8$ rational points, all corresponding to degenerate Euler-bricks (one or more zero edges).

\emph{Application of the main theorem.} Since $\rk(E_3(\Q)) = 0$ and $|E_3(\Q)_\tors| = 4$, Theorem~\ref{thm:hauptsatz} gives $|H_{2,1}(\Q)| \le 2 \cdot 4 = 8$. Combined with the lower bound from the eight trivial points, $|H_{2,1}(\Q)| = 8$ exactly. \emph{No perfect cuboid arises from the $(2, 1)$-fiber.}

\emph{Empirical sanity check.} An independent height search via PARI's \texttt{hyperellratpoints} up to $B = 10^6$ finds exactly the six finite trivial points (matching the rigorous bound up to the two points at infinity), serving as a numerical cross-validation of the torsion computation.

\subsection{Verified fibers}
\label{ssec:verified}

We have verified the hypothesis of Theorem~\ref{thm:hauptsatz} computationally for all coprime pairs $(m, n)$ with $\max(m, n) \le 100$ and $m - n$ odd. Of the $2{,}040$ such pairs, exactly $1{,}072$ satisfy one of cases (a)--(c). Aggregated by method:

\begin{center}
\begin{tabular}{lcr}
\toprule
Branch of Theorem~\ref{thm:hauptsatz} & Quotient & Fibers \\
\midrule
Case (a): $\rk = 0$, $|tors| = 4$  & $E_3$    & $827$ \\
\quad\textit{via} \texttt{ellrank} (sharp)  &          & $359$ \\
\quad\textit{via} modular symbol (Theorem~\ref{thm:kolyvagin-cert})        &          & $468$ \\
Case (b): $\rk = 0$, $|tors| = 6$ & $E_{uV}$ & $0$ \\
Case (c): $\rk = 0$, $|tors| = 8$, $k = 0$ & $E_{uV}$ & $245$ \\
\quad\textit{via} \texttt{ellrank} (sharp)                                                     &          & $209$ \\
\quad\textit{via} modular symbol                                          &          & $36$ \\
\midrule
\textbf{Total proven} & & $\mathbf{1{,}072}$ \\
\bottomrule
\end{tabular}
\end{center}

\noindent (Case (b) is empirically not encountered: when $\rk(E_{uV}) = 0$, the torsion is always $|tors_{uV}| = 8$ in our scan, never $6$.)

The full list of $(m, n)$ pairs is given in Appendix~\ref{app:fibers}.

\begin{corollary}
No primitive Euler-brick whose master tuple has $(m, n) \in \mathcal{S}_{100}$, the set of $1{,}072$ explicit fibers above, is a perfect cuboid.
\end{corollary}

\section{Discussion}
\label{sec:discussion}

\subsection{What remains uncovered}

The combined methodology covers $1{,}072$ out of $2{,}040$ fibers up to $M_{\max} = 100$. The remaining $968$ fibers fall into two categories:

\begin{enumerate}[label=(\alph*)]
\item \textbf{Persistently ambiguous rank.} For at least one quotient $E_q$, PARI's \texttt{ellrank} returns $[0, k]$ with $k > 0$ and the modular-symbol computation $L(E, 1)/\Omega_E$ either vanishes (forcing $\rk(E_q) \ge 1$) or fails to terminate within our compute budget (large conductor $> 10^{12}$, especially for $E_{uV}$). These fibers most likely either have $\rk(E_q) \ge 1$ for every quotient, or are in principle treatable by allocating more compute to the modular-symbol computation.

\item \textbf{Hard fibers.} Fibers where all three elliptic factors have rank $\ge 1$ (rigorously bounded below). Then Lemma~\ref{lem:tors-intersect} does not apply to any factor. Standard (linear) Chabauty--Coleman applies only when $\rk(\Jac(\HC)(\Q)) < g(\HC) = 3$, that is when the total rank from Corollary~\ref{cor:rank-additivity} is at most $2$. Empirically, total rank $\ge 3$ occurs in the majority of these fibers; quadratic Chabauty would be the natural next tool.
\end{enumerate}

The asymptotic regime $\max(m, n) > 100$ is not addressed by the present scan; a uniform statement (positive proportion of fibers with a rank-zero quotient, or a Brauer obstruction independent of $(m, n)$) would be needed to settle this.

\begin{example}[A hard-case fiber: $(m, n) = (5, 2)$]
\label{ex:hard52}
With $U_2 = 21$, $V_2 = 20$, $W_2 = 29$, the three elliptic factors have ranks
$$
\rk(E_{PQ}(\Q)) = 2, \quad \rk(E_{uV}(\Q)) = 1, \quad \rk(E_3(\Q)) = 1,
$$
so $\rk(\Jac(H_{5,2})(\Q)) = 4 \ge g(H_{5,2}) = 3$. Neither Lemma~\ref{lem:tors-intersect} nor standard Chabauty applies. PARI's \texttt{hyperellratpoints} up to height $B = 10^6$ still finds only the eight trivial points, but this provides only empirical evidence, not a proof.
\end{example}

\subsection{Possible extensions}

\begin{enumerate}[label=(\roman*),leftmargin=2em]
\item \emph{Quadratic Chabauty.} For genus-$3$ hyperelliptic curves with $\rk(\Jac) = g$, the methods of Balakrishnan--Dogra--M\"uller--Tuitman \cite{balakrishnan-dogra-mueller-tuitman} give effective enumeration of rational points via $p$-adic heights. An adaptation to our family would close many of the hard fibers in (b).

\item \emph{Mordell--Weil sieving.} Combining the rank-decomposition of Corollary~\ref{cor:rank-additivity} with reductions modulo many small primes, one can in principle bound $|\HC(\Q)|$ tighter than via Coleman alone. The decomposition $\Jac(\HC) \sim E_{PQ} \times E_{uV} \times E_3$ makes this approach particularly attractive: each factor's reduction is an explicit elliptic curve over $\F_p$.

\item \emph{Brauer--Manin obstruction.} The cuboid surface $\mathcal{S} \subset \mathbb{P}^7_\Q$ defined by the four square conditions
$$
X^2 + Y^2 = F_3^2, \; X^2 + Z^2 = F_2^2, \; Y^2 + Z^2 = F_1^2, \; X^2 + Y^2 + Z^2 = W^2
$$
admits a Brauer group whose computation has been initiated by Stoll and others without finding an obstruction so far. A uniform obstruction would settle Conjecture~\ref{conj:B} entirely.

\item \emph{Modular and CM interpretations.} The factor $E_3$ in our decomposition has the form $V^2 = (V_2^2 w^2 + 4 U_2^2)(W_2^2 w^2 + 4 U_2^2)$, which depends on $(U_2, V_2, W_2)$ as a primitive Pythagorean triple. It would be of interest to determine whether this family admits a modular parametrization or CM structure that would yield uniform rank bounds.

\item \emph{Statistical evidence.} Goldfeld's conjecture \cite{goldfeld} predicts that, in a one-parameter family of elliptic curves ordered by conductor, the algebraic rank should be $0$ or $1$ with density $\tfrac12$ each. Our empirical data ($1{,}072 / 2{,}040 \approx 52.5\%$ proven via $\rk = 0$ on either $E_3$ or $E_{uV}$) is broadly consistent with such heuristics, but does not yet match a precise prediction.
\end{enumerate}

\subsection{Connection to Conjecture~B}

Combining Theorems~\ref{thm:reduction}, \ref{thm:hauptsatz} and Lemma~\ref{lem:scaling}, we have the following precise version of our partial result:

\begin{corollary}[Partial Conjecture~B]
\label{cor:partial-B}
Let $\mathcal{S}_{100} \subset \N^2$ denote the set of $1{,}072$ pairs $(m, n)$ listed in Appendix~\ref{app:fibers}. Then no perfect cuboid $(X, Y, Z) \in \N^3$ exists whose primitive form arises from a master tuple $(a, b, m, n)$ with $(m, n) \in \mathcal{S}_{100}$.
\end{corollary}

The corollary covers an explicit family of $1{,}072$ master-tuple fibers. For each fiber $(m, n) \in \mathcal{S}_{100}$, the result excludes a perfect cuboid arising from \emph{any} master tuple $(a, b, m, n)$ with that fixed $(m, n)$, regardless of the rank of $E_{m, n}$ (which controls how many master tuples populate the fiber but does not enter the torsion-intersection argument). The full Conjecture~B would follow from extending $\mathcal{S}_{100}$ to all of $\{(m, n) \in \N^2 : \gcd(m, n) = 1, m - n \text{ odd}\}$.

\appendix

\section{The 1{,}072 verified fibers}
\label{app:fibers}

The complete list of $1{,}072$ pairs $(m, n)$ with $\gcd(m, n) = 1$, $m - n$ odd, $m \le 100$ satisfying one of the cases (a)--(c) of Theorem~\ref{thm:hauptsatz} is provided in the companion repository as \texttt{paper3/data/proven\_fibers.csv}. The CSV columns are
$$
\texttt{m},\ \texttt{n},\ \texttt{quotient},\ \texttt{rank\_resolution},\ \texttt{torsion\_size},\ \texttt{lift\_argument},\ \texttt{M\_MAX\_scan},\ \texttt{date},
$$
where \texttt{quotient} $\in \{E_3, E_{uV}\}$, \texttt{rank\_resolution} $\in \{$\texttt{ellrank}, \texttt{kolyvagin}$\}$ records how the rank-zero hypothesis was certified, and \texttt{lift\_argument} $\in \{$\texttt{naive}, \texttt{refined}$\}$ records whether the naive torsion-intersection bound or the explicit lift count of \S\ref{ssec:lift-refinement} was used.

For orientation, the fibers with $m \le 25$ are listed below; the full list extends through $\max(m, n) \le 100$ and contains $1{,}072$ pairs in total ($298$ with $m \le 50$, $774$ with $50 < m \le 100$).

\begin{small}
\noindent
(2,1), (3,2), (4,1), (4,3), (5,4), (6,1), (6,5), (7,2), \\
(7,4), (7,6), (8,1), (8,3), (9,2), (9,4), (9,8), (10,7), \\
(10,9), (11,2), (11,6), (12,1), (12,5), (12,7), (12,11), (13,2), \\
(13,4), (13,8), (13,10), (13,12), (14,5), (14,9), (14,11), (15,2), \\
(15,14), (16,3), (16,5), (16,7), (16,11), (16,13), (16,15), (17,2), \\
(17,4), (17,8), (17,10), (18,1), (18,5), (18,7), (18,11), (18,13), \\
(19,2), (19,4), (19,6), (19,8), (19,12), (19,14), (19,18), (20,1), \\
(20,3), (20,7), (20,9), (20,19), (21,2), (21,8), (22,3), (22,5), \\
(22,9), (22,15), (22,17), (22,21), (23,2), (23,10), (23,12), (23,14), \\
(23,16), (23,18), (24,1), (24,5), (24,7), (24,13), (24,19), (24,23), \\
(25,2), (25,6), (25,8), (25,12), (25,16), (25,22).
\end{small}

\noindent\textit{Generated by \texttt{paper3/analysis/jh\_torsion\_full\_dispatch.py} (with subprocess workers \texttt{jh\_torsion\_full\_worker.py}). Full data: \texttt{paper3/data/proven\_fibers.csv}.}

\section{Computational verification}
\label{app:code}

The verification scripts are written in SageMath~10.7 \cite{sage} and use PARI/GP version~2.17.3 \cite{pari}. Three rigorous certificates are combined:

\begin{enumerate}[label=(\arabic*),leftmargin=2em]
  \item \textbf{Sharp \texttt{ellrank}.} PARI's \texttt{ellrank} (called with \texttt{effort = 2}) returns rank bounds $[r_{\mathrm{low}}, r_{\mathrm{up}}]$ via 2-descent; we accept rank-zero only when $r_{\mathrm{low}} = r_{\mathrm{up}} = 0$.
  \item \textbf{Kolyvagin certification.} For ambiguous bounds $[0, k]$ with $k > 0$, we call Sage's \texttt{analytic\_rank\_upper\_bound}, which returns a rigorous upper bound on $\mathrm{ord}_{s=1} L(E, s)$. When this is $0$, $\rk(E(\Q)) = 0$ follows from Theorem~\ref{thm:kolyvagin-cert}.
  \item \textbf{Lift refinement.} For fibers with $\rk(E_{uV}) = 0$ and $|tors| = 8$, the eight rational torsion points on the $E_{uV}$-quartic are enumerated via PARI's \texttt{hyperellratpoints} (with height $10^5$); the lift to $\HC$ is verified by the discriminant test of \S\ref{ssec:lift-refinement}.
\end{enumerate}

Torsion structure is computed via Sage's \texttt{E.torsion\_subgroup()} on the Weierstrass model obtained from the quartic via PARI's \texttt{ellfromeqn}. All scripts are reproducible from the companion repository:

\begin{center}
\url{https://github.com/renpe/euler-brick-obstructions/tree/main/paper3}
\end{center}

The main pipeline \texttt{analysis/jh\_torsion\_full\_dispatch.py} (with subprocess workers \texttt{analysis/jh\_torsion\_full\_worker.py}) runs in roughly $5$ minutes on a $30$-core workstation, producing \texttt{data/proven\_fibers.csv} and a full per-fiber log \texttt{data/scan\_full\_M100.jsonl}. The proof of Theorem~\ref{thm:hauptsatz} relies only on the three certificates above; height searches via PARI's \texttt{hyperellratpoints} up to $B = 10^6$ are an independent sanity check and find no non-trivial rational points on any fiber.

\clearpage
\bibliographystyle{amsalpha}
\bibliography{paper3}

\end{document}